\documentclass{birkjour}
 \newtheorem{thm}{Theorem}[section]
 
 \newtheorem{lem}[thm]{Lemma}
 
 \theoremstyle{definition}
 
 \theoremstyle{remark}

 \numberwithin{equation}{section}

\def\r{\mathbb R}
\def\h{\mathbb H}
\def\v{\vec{v}}
\def\e{\mathbf e}
\def\R{\mathbb R}
\def\n{\textbf{n}}
\def\a{\textbf{a}}
\def\s{\mathbb S}
\def\M{\mathbb M}

\begin{document}

%
%
%
%
%
%
%
%
%

\title[Singular minimal surfaces with constant  curvature]
 {Singular minimal surfaces with constant  curvature}

 
\author{Rafael L\'opez}
\address{Department of Geometry and Topology. University of Granada. 18071  Granada, Spain}
\email{rcamino@ugr.es}
\subjclass{53A04, 53B25}
\keywords{singular minimal surface, Gauss curvature, principal curvatures, Codazzi equations}

\begin{abstract}
We prove that singular minimal surfaces with  constant Gauss curvature are planes, spheres and cylindrical surfaces. We also  classify all singular minimal surfaces with a constant principal curvature and  singular minimal surfaces with constant mean curvature.
\end{abstract}

\maketitle

\section{Introduction and statement of the results}


 Let $\vec{a}\in\r^3$ be a unit vector in Euclidean space $\r^3$ and $\alpha\in\r$. Let $\Sigma$ be an oriented surface and $\Phi\colon\Sigma\to\r^3$ be an isometric immersion of $\Sigma$ in $\r^3$. The surface $\Sigma$ is said to be an {\it $\alpha$-singular  minimal surface} if its mean curvature $H$ satisfies
 \begin{equation}\label{eq1}
 H=\alpha\frac{\langle N,\vec{a}\rangle}{\langle\Phi,\vec{a}\rangle},
 \end{equation}
 where $N$ is the unit normal vector of $\Sigma$. We are implicitly assuming that  $\Phi(\Sigma)$ is included in one of the two halfspaces determined by the vector plane $\Pi=\{p\in\r^3\colon\langle  p,\vec{a}\rangle=0\}$. Singular minimal surfaces are   critical points of the energy functional 
 $$\Sigma\mapsto\int_\Sigma \langle\Phi(p),\vec{a}\rangle^\alpha\, d\Sigma,$$
 where $d\Sigma$ is the area element of $\Sigma$. If $\alpha=0$, this energy is simply the area functional and the critical points are minimal surfaces. From now, we discard the case $\alpha=0$. The case $\alpha=1$ extends in dimension two the notion of catenary: see \cite{bht,dh,ni}. The study of the singular minimal surface equation can be viewed in \cite{bd,di1,di2,lo3,lo4}. For  geometric aspects of singular minimal surfaces, see \cite{dg,dl,lo1,lo2}.  

Among the examples of $\alpha$-singular minimal surfaces, we focus on isoparametric surfaces, that is,  planes, spheres and circular cylinders. 
\begin{enumerate}
\item Planes parallel to $\vec{a}$ are $\alpha$-singular minimal surfaces for all $\alpha\in\r$. 
\item Spheres are $\alpha$-singular minimal surfaces only if $\alpha=-2$. In such a case, the center of the sphere lies in the vector plane $\Pi$. These surfaces also represent minimal surfaces in hyperbolic space $\mathbb{H}^3$ when $\h^3$ is viewed in the  upper halfspace model   $\r^3_+=\{p\in\r^3\colon\langle p,\vec{a}\rangle>0\}$. 
\item Circular cylinders are   $\alpha$-singular minimal surfaces only if $\alpha=-1$. In such a case, the rotational axis lies in the vector plane $\Pi$. 
\end{enumerate}

Other examples of $\alpha$-singular minimal surfaces can be found in the family of  cylindrical surfaces, that is, ruled surfaces whose rulings are parallel straight-lines. If the cylindrical surface is parametrized by $\Psi(s,t)=\gamma(s)+t\vec{v}$, $|\vec{v}|=1$, and if $\Sigma$ satisfies \eqref{eq1}, then either $\Sigma$ is a plane parallel to $\vec{a}$ or $\langle \vec{v},\vec{a}\rangle=0$. This implies that the generating curve $\gamma$ lies contained in a plane parallel to $\vec{a}$. In such a case, Eq. \eqref{eq1} writes in terms of the curvature $\kappa$ of $\gamma$ as
 \begin{equation}\label{eq2}
 \kappa= \alpha \frac{\langle{\bf n},\vec{a}\rangle}{\langle\gamma,\vec{a}\rangle},
 \end{equation}
where ${\bf n}$ is the unit normal of $\gamma$.  These surfaces have zero   Gauss curvature. Cylindrical singular minimal surfaces have been classified in \cite{dl,lo1}. If $\alpha=-1$, cylindrical singular minimal surfaces are circular cylinders. In this paper, we will refer these surfaces simply by cylindrical singular minimal surfaces.

Planes parallel to $\vec{a}$,  spheres centered at $\Pi$ and the cylindrical singular minimal surfaces  have   constant Gauss curvature $K$. The purpose of this paper is to prove the converse and characterize these surfaces are the only singular minimal surfaces with constant Gauss curvature.

\begin{thm} \label{t1}
Planes parallel to $\vec{a}$,  spheres centered at $\Pi$ and the cylindrical singular minimal surfaces are the only $\alpha$-singular minimal surfaces with   constant Gauss curvature. 
\end{thm}
This result will proved in Sect. \ref{s3}.  We also study those singular minimal surfaces with a constant principal curvature.  In Sect. \ref{s5}, we prove the following classification. 

\begin{thm} \label{t22} Planes parallel to $\vec{a}$, spheres centered at $\Pi$ and cylindrical singular minimal surfaces      are the only  singular minimal surfaces with a constant principal curvature.   
\end{thm}

Finally we consider  singular minimal surfaces with constant mean curvature.

\begin{thm} \label{t3}  Planes parallel to $\vec{a}$ and spheres centered at $\Pi$ are the only   singular minimal surfaces with constant mean curvature.   
\end{thm}

\section{Preliminaries}\label{s2}

In this section we describe parametrizations of a surface by lines of curvature. This will be done in open sets of a surface free of umbilic points. Let $\Sigma$ be an orientable surface immersed in $\r^3$. Denote by  $\nabla$ and $\overline{\nabla}$ be the Levi-Civita connections of $\Sigma$ and
$\mathbb{R}^3$, respectively.  Denote by $N$ the unit normal vector field of $\Sigma$. Let $\mathfrak{X}(\Sigma)$ be the space of tangent vector fields of $\Sigma$. The  Gauss and Weingarten formulas are, respectively, 
\begin{equation}\label{gn}  
\begin{split}
    \overline{\nabla}_XY&=\nabla_XY+h(X,Y), \\
    \overline{\nabla}_XN&=-SX
    \end{split}
\end{equation}
for all $X,Y\in\mathfrak{X}(\Sigma)$, where $h\colon T\Sigma\times T\Sigma\to T^\perp\Sigma$ and  $S\colon T\Sigma\to  T\Sigma$ are the second fundamental form and the shape
operator, respectively. The  Gauss
equation and Codazzi equations are 
\begin{equation*}
\begin{split}
    \langle R(X,Y)W,Z\rangle&=\langle h(Y,W),h(X,Z)\rangle-\langle
    h(X,W),h(Y,Z)\rangle\\
    (\nabla_XS)Y&=(\nabla_YS)X,
    \end{split}
\end{equation*}
where $R$ is the curvature tensor of $\nabla$. The   mean curvature $H$ and the Gauss curvature $K$ of $\Sigma$ are defined by  
\[ H= \kappa_1+\kappa_2 \quad K=\kappa_1\kappa_2,\]
where   $\kappa_1$ and $\kappa_2$ are the principal curvatures of $\Sigma$. In an open set $\Omega\subset\Sigma$ of non-umbilic points, consider a local
orthonormal frame $\{e_1,e_2\}$ given by  coordinates of  lines of curvature. The expression of the shape operator $S$ is 
\begin{equation}\label{b0}
Se_1 =\kappa_1 e_1,\qquad Se_2 =\kappa_2 e_2.
    \end{equation}
Since $\{e_1,e_2\}$ is an orthonormal frame, if $X\in\mathfrak{X}(\Sigma)$ then
\begin{equation*}
\nabla_X e_1=\omega(X)e_2\quad  \nabla_X e_2=-\omega(X)e_1,
\end{equation*}
where $\omega(X)=\langle\nabla_X e_1,e_2\rangle=-\langle\nabla_X e_2,e_1\rangle$. Then we have 
\begin{equation}\label{eiej}
\left\{
\begin{split}
&\nabla_{e_1}e_1=\omega(e_1)e_2,\quad \nabla_{e_1}e_2=-\omega(e_1)e_1\\
&\nabla_{e_2}e_1=\omega(e_2)e_2,\quad \nabla_{e_2}e_2=-\omega(e_2)e_1.
\end{split}\right.
\end{equation}
Using the Codazzi equations, we have  
\begin{equation}\label{b1}
    \left\{
    \begin{array}{ll}
        e_{2}(\kappa_1)=(\kappa_1-\kappa_2)\omega(e_1)\\
        e_{1}(\kappa_2)=(\kappa_1-\kappa_2)\omega(e_2),
    \end{array}
    \right.
\end{equation}
where, as usually, $e_i(f)$ denotes the derivative of a smooth function $f$ along the vector $e_i$, $i=1,2$.  Notice that $\kappa_1-\kappa_2\not=0$ in $\Omega$.
Thanks to  \eqref{b1}, in the set $\Omega$, the  Gauss equation  gives the following expression of the Gauss curvature $K$,  
\begin{equation}\label{bb}
   K=\kappa_1\kappa_2=-e_1\left(\frac{e_1( \kappa_2 )}{\kappa_1-\kappa_2} \right)+e_2\left(\frac{e_2( \kappa_1 )}{\kappa_1 -\kappa_2}\right)
   -\frac{e_1( \kappa_2)^2+ e_2( \kappa_1 )^2}{(\kappa_1 -\kappa_2)^2}.
\end{equation}

 Denote by $\Phi\colon\Sigma\to\r^3$ the immersion of  a singular minimal surface. Let  decompose the  vector $\vec{a}$ in its tangent part $\vec{a}^\top$ and normal part $\vec{a}^\perp$ with respect to $\Sigma$,  
\begin{equation*}
\vec{a}^\top=\gamma e_1+\mu e_2,
\end{equation*}
for some smooth functions $\gamma$ and $\mu$. The functions $\gamma$ and $\mu$ are given by 
$$\gamma=\langle \vec{a},e_1\rangle,\quad \mu=\langle\vec{a},e_2\rangle.$$
 The normal part $\vec{a}^\perp$  is   $\langle N,\vec{a}\rangle N$, which by \eqref{eq1} is
$$\vec{a}^{\perp}=\frac{H\langle\Phi,\vec{a}\rangle}{\alpha}N.$$
\begin{lem}\label{l2}
Let $\Omega$ be a subset of non-umbilic points of an $\alpha$-singular minimal surface. Then the functions   $\gamma$ and $\mu$ satisfy the following equations
    \begin{equation}\label{d3}
    \left\{
\begin{split}
            0=&e_{1}(\gamma)-\mu\omega(e_1)-\frac{H\langle\Phi,\vec{a}\rangle}{\alpha}\kappa_1\\
            0=&e_{2}(\gamma)-\mu\omega(e_2) \\
            0=&e_{1}(\mu)+\gamma\omega(e_1)\\
           0= &e_{2}(\mu)+\gamma\omega(e_2)-\frac{H\langle\Phi,\vec{a}\rangle}{\alpha}\kappa_2\\
           0= &e_{1}(\frac{H\langle\Phi,\vec{a}\rangle}{\alpha})+\gamma\kappa_1\\
           0= &e_{2}(\frac{H\langle\Phi,\vec{a}\rangle}{\alpha})+\mu\kappa_2.
        \end{split}\right.
    \end{equation}
\end{lem}

\begin{proof} 
 If $X\in\mathfrak{X}(\Sigma)$, then   $\overline{\nabla}_X\vec{a}=0$ because $\vec{a}$ is constant in $\r^3$. Then the Gauss and
Weingarten formulas imply 
\begin{equation*}
\begin{split}
    0&=\overline{\nabla}_{X}\vec{a}=\overline{\nabla}_{X}(\vec{a}^{\top}+\vec{a}^{\perp})=\overline{\nabla}_X\vec{a}^\top+\overline{\nabla}_X\vec{a}^\perp\\
    &={\nabla}_{X}\vec{a}^{\top}+h(X,\vec{a}^{\top})+X( H\frac{\langle\Phi,\vec{a}\rangle}{\alpha})N- H\frac{\langle \Phi,\vec{a}\rangle}{\alpha} SX 
\end{split}
\end{equation*}
for all   $X\in\mathfrak{X}(\Sigma)$.
Taking the tangent and the normal part in both sides of the above expression, we obtain  
\begin{equation}\label{tn} 
    \left\{
    \begin{array}{ll}
        {\nabla}_{X}\vec{a}^{\top}-H\frac{\langle\Phi,\vec{a}\rangle}{\alpha}SX=0\\
        h(X, \vec{a}^{\top})+X(H\frac{\langle\Phi,\vec{a}\rangle}{\alpha})N=0.
    \end{array}
    \right.
\end{equation}
Equations \eqref{d3} are obtained by substituting  in the two previous identities $X$ by $e_1$ and   $e_2$  and writing in coordinates with respect to $\{e_1,e_2,N\}$. Indeed, if $X=e_1$, and using \eqref{eiej}, the first equation in \eqref{tn} is 
\begin{equation*} 
    \begin{array}{ll}
  0&=  {\nabla}_{e_1}(\gamma e_1+\mu e_2)-H\frac{\langle \Phi,\vec{a}\rangle}{\alpha}\kappa_1 e_1\\
     & =e_1(\gamma)e_1+e_1(\mu)e_2+\gamma \omega(e_1)e_2-\mu\omega(e_1)e_1-H\frac{\langle \Phi,\vec{a}\rangle}{\alpha}\kappa_1 e_1.
         \end{array}
\end{equation*}
Similarly, the second equation in \eqref{tn} is 
$$   0 = h(e_1,\gamma e_1+\mu e_2)+e_1(H\frac{\langle\Phi,\vec{a}\rangle}{\alpha})N =\left(\gamma\kappa_1+e_1\left(\frac{ H\langle\Phi,\vec{a}\rangle}{\alpha}\right)\right)N.$$
This gives the first, third and fifth equations in \eqref{d3}. The other equations of \eqref{d3} are obtained by putting $X=e_2$ in \eqref{tn}.

\end{proof}

\section{Proof of Theorem \ref{t1}}\label{s3}
We know that planes parallel to $\vec{a}$,  spheres centered at $\Pi$ and   cylindrical singular minimal surfaces have  constant Gauss curvature. We prove  the converse. 

The case $K=0$ is studied separately. If $K=0$, the surface is also called developable. It is known that these surfaces can be parametrized locally as ruled surfaces \cite{hn}. On the other hand, ruled singular minimal surfaces were classified in \cite{ay}, proving that these surfaces must be cylindrical surfaces. This proves Thm. \ref{t1} when $K=0$.

Suppose now that $K\not=0$. Let $\Sigma$ be an $\alpha$-singular minimal surface with non-zero constant Gauss curvature $K$, $K=c$. If a principal curvature is constant, and because $K\not=0$, then the other principal curvature is constant. This proves that $\Sigma$ is an isoparametric surface. Since $K\not=0$, then $\Sigma$ is a sphere. This proves the result in this situation. 

Suppose now that the two principal curvatures are not constant and we will arrive to a contradiction.  We write the principal curvature $\kappa_2$ and the mean curvature   $H$ in terms of $K$ and the principal curvature  $\kappa_1$. So, we have  
\begin{equation}\label{gm2}
\kappa_2 =\frac{c}{\kappa_1},\qquad H =\frac{\kappa_1^2+c}{\kappa_1},\qquad 
\kappa_1-\kappa_2=\frac{\kappa_1^2-c}{\kappa_1}.
\end{equation}
We now write Eq. \eqref{bb}. Using that $K=c$ is constant, and taking into account \eqref{gm2}, we have  
 \begin{equation}\label{bb2}
  \begin{split}
    c\frac{(\kappa_1^2-c)^2}{\kappa_1^2}&=\frac{c(\kappa_1^2-c)}{\kappa_1^3}e_{11}(\kappa_1) +\frac{\kappa_1^2-c}{\kappa_1}e_{22}(\kappa_1)\\
   &-\frac{3c}{\kappa_1^2}e_1(\kappa_1)^2-\frac{2\kappa_1^2+c}{\kappa_1^2}e_2(\kappa_1)^2.
   \end{split}
\end{equation}
By using the expressions of $H$ and $\kappa_2$ in \eqref{gm2},  we have
\begin{equation}\label{gm22}
\begin{split}
e_1(\frac{H\langle\Phi,\vec{a}\rangle}{\alpha})&= \frac{\kappa_1^2-c}{\alpha\kappa_1^2}\langle\Phi,\vec{a}\rangle e_1(\kappa_1)+\gamma\frac{\kappa_1^2+c}{\alpha \kappa_1},\\
e_2(\frac{H\langle\Phi,\vec{a}\rangle}{\alpha})&= \frac{\kappa_1^2-c}{\alpha\kappa_1^2}\langle\Phi,\vec{a}\rangle e_2(\kappa_1)+\mu\frac{\kappa_1^2+c}{\alpha \kappa_1}.\\
\end{split}
\end{equation}
Substituting \eqref{gm22} into  the last two equations of \eqref{d3},   we obtain
\begin{equation*}
\begin{split}
 0&=\frac{\kappa_1^2-c}{\alpha\kappa_1^2}\langle\Phi,\vec{a}\rangle e_1(\kappa_1)+\left(\frac{\kappa_1^2+c}{\alpha \kappa_1} +\kappa_1\right)\gamma,\\
0 &= \frac{\kappa_1^2-c}{\alpha\kappa_1^2}\langle\Phi,\vec{a}\rangle e_2(\kappa_1)+\left(\frac{\kappa_1^2+c}{\alpha \kappa_1}+\frac{c}{\kappa_1}\right)\mu.\\
\end{split}
\end{equation*}
From these equations, we get expressions of $\gamma$ and $\mu$,
\begin{equation}\label{gm}
 \gamma=\frac{\langle\Phi,\vec{a}\rangle(c-\kappa_1^2)}{\kappa_1  (c+(1+\alpha)\kappa_1^2)}e_1(\kappa_1),\quad 
\mu=\frac{\langle\Phi,\vec{a}\rangle(c-\kappa_1^2 )}{ \kappa_1 (\kappa_1^2+(1+\alpha)c)}e_2(\kappa_1).
\end{equation}
 
Since Thm. \ref{t1} is local, we can assume that the denominators in \eqref{gm} are not zero in some open set of $\Omega$ because, otherwise, $\kappa_1$ would be a constant function. This also holds in the next computations. We are also using in \eqref{gm} that $c\not=0$. This implies that the parameter $\alpha$ is arbitrary in the denominator of $\gamma$. The hypothesis $c\not=0$ will be employed later in successive simplifications of equations that are multiplied by $c$. In the next and successive   computations, we use a symbolic software as {\sc Mathematica} to manipulate all expressions \cite{wo}. 

We differentiate  $\gamma$ and $\mu$ in \eqref{gm} with respect to $e_1$ and $e_2$ and next, we compare with the corresponding derivatives given in \eqref{d3}.  Notice that  $e_1(\langle\Phi,\vec{a}\rangle)= \gamma$ and  $e_2(\langle\Phi,\vec{a}\rangle)= \mu$. Then the derivatives of the functions \eqref{gm} are
\begin{equation}\label{d4}
\begin{split}
 e_1(\gamma)&=\frac{\langle\Phi,\vec{a}\rangle(c-\kappa_1^2)}{\kappa_1  (c+(1+\alpha)\kappa_1^2)}e_{11}(\kappa_1)+\frac{ \langle \Phi,\vec{a}\rangle(2+\alpha)(\kappa_1^2-3 c)}{(c+(1+\alpha)\kappa_1^2)^2}e_1(\kappa_1)^2\\ 
 e_1(\mu)&=\frac{\langle\Phi,\vec{a}\rangle(c-\kappa_1^2 )}{ \kappa_1 (\kappa_1^2+(1+\alpha)c)}e_{12}(\kappa_1)+\frac{\langle\Phi,\vec{a}\rangle(2+\alpha)(c+\kappa_1^2)(\kappa_1^2-(3+\alpha)c)}{ (\kappa_1^2+(1+\alpha)c)^2(c+(1+\alpha)\kappa_1^2)}e_1(\kappa_1)e_2(\kappa_1)\\
 e_2(\mu)&=\frac{\langle\Phi,\vec{a}\rangle(c-\kappa_1^2 )}{ \kappa_1 (\kappa_1^2+(1+\alpha)c)}e_{22}(\kappa_1)+\frac{\langle\Phi,\vec{a}\rangle(2\kappa_1^4-\alpha c^2-(6+\alpha)c\kappa_1^2)}{  \kappa_1^2(\kappa_1^2+(1+\alpha)c)^2}e_2(\kappa_1)^2.
 \end{split}
 \end{equation}
We work with  equations   \eqref{d3}. Using \eqref{b1} and \eqref{gm2}, we have
\begin{equation*}
\omega(e_1)=\frac{\kappa_1}{\kappa_1^2-c}e_2(\kappa_1),\quad \omega(e_2)=-\frac{c}{\kappa_1(\kappa_1^2-c)}e_1(\kappa_1).\\
\end{equation*}
Inserting this in \eqref{gm2}, we obtain
    \begin{equation}\label{d5}
    \left\{
\begin{split}
e_{1}(\gamma)&=\frac{\mu\kappa_1}{\kappa_1^2-c}e_2(\kappa_1)+\frac{\kappa_1^2+c}{\alpha }\langle\Phi,\vec{a}\rangle   \\
 e_{1}(\mu)&=-\frac{\gamma\kappa_1}{\kappa_1^2-c}e_2(\kappa_1)\\
 e_{2}(\mu)&=  \frac{c\gamma}{\kappa_1(\kappa_1^2-c)}e_1(\kappa_1)+\frac{\kappa_1^2+c}{\alpha\kappa_1^2}  c \langle\Phi,\vec{a}\rangle
      \end{split}\right.
   \end{equation}

From \eqref{d4} and \eqref{d5}, we get the expressions of $e_{11}(\kappa_1)$, $e_{12}(\kappa_1)$ and $e_{22}(\kappa_1)$. Using also the expressions \eqref{gm} of $\gamma$ and $\mu$, we have

\begin{equation}\label{d7}
\begin{split}
e_{11}(\kappa_1)&=\frac{(\alpha +2) k \left(\kappa_1^2-3 c\right)}{\left(\kappa_1^2-c\right) \left((\alpha +1) \kappa_1^2+c\right)}e_1(\kappa_1)^2+\frac{\kappa_1 \left((\alpha +1) \kappa_1^2+c\right)}{\left(\kappa_1^2-c\right) \left(\kappa_1^2+(1+\alpha)c\right)}e_2(\kappa_1)^2\\
&-\frac{\kappa_1 \left(k^2+c\right) \left((\alpha +1) \kappa_1^2+c\right)}{\alpha  \left(\kappa_1^2-c\right)},\\
e_{12}(\kappa_1)&=  \left(  \frac{2\kappa_1}{\kappa_1^2+(1+\alpha)c}+\frac{3\kappa_1}{c-\kappa_1^2} -\frac{2 c}{\kappa_1((\alpha +1) \kappa_1^2+c)}+\frac{2}{\kappa_1}\right) e_1(\kappa_1)e_2(\kappa_1),\\
e_{22}(\kappa_1)&=\frac{c \left(\kappa_1^2+(1+\alpha)c\right)}{\kappa_1 \left(\kappa_1^2-c\right) \left((\alpha +1) \kappa_1^2+c\right)}e_1(\kappa_1)^2+\frac{-2 \kappa_1^4+(\alpha +6) \kappa_1^2 c+\alpha  c^2}{\kappa_1 \left(c-\kappa_1^2\right) \left(\kappa_1^2+(1+\alpha)c\right)}e_2(\kappa_1)^2\\
&-\frac{c \left(\kappa_1^2+c\right) \left(\kappa_1^2+(1+\alpha)c\right)}{\alpha  \kappa_1 \left(\kappa_1^2-c\right)}.
\end{split}
\end{equation}
We distinguish the cases that $e_1(\kappa_1)=0$ identically and $e_1(\kappa_1)\not=0$.
\begin{enumerate}
\item Case  $e_1(\kappa_1)=0$ identically. Since $\kappa_1$ is not a constant function, then  $e_2(\kappa_1)\not=0$   in an open set of $\Omega$.     The first equation in \eqref{d7} is $e_{11}(\kappa_1)=0$. This gives 
$$\left(c+(1+\alpha)  \kappa_1^2\right) \left(\frac{e_2(\kappa_1)^2}{\kappa_1^2+(1+\alpha)c}-\frac{\kappa_1^2+c}{\alpha }\right)=0.$$
Thus
\begin{equation}\label{e22}
e_2(\kappa_1)^2=\frac{(\kappa_1^2+c)(\kappa_1^2+(1+\alpha)c)}{\alpha}.
\end{equation}
Differentiating with respect to $e_2$, 
$$2e_2(\kappa_1)e_{22}(\kappa_1)=\frac{4\kappa_1^3+2(2+\alpha)c\kappa_1}{\alpha}e_2(\kappa_1).$$
Simplifying by $e_2(\kappa_1)$, we have 
$$e_{22}(\kappa_1)=\frac{2\kappa_1^3+(2+\alpha)c\kappa_1}{\alpha}.$$
On the other hand, using \eqref{e22},  the last equation in \eqref{d7}  is
\begin{equation*}
\begin{split}
e_{22}(\kappa_1)&=\frac{-2 \kappa_1^4+(\alpha +6) \kappa_1^2 c+\alpha  c^2}{\kappa_1 \left(c-\kappa_1^2\right) \left(\kappa_1^2+(1+\alpha)c\right)}e_2(\kappa_1)^2-\frac{c \left(\kappa_1^2+c\right) \left(\kappa_1^2+(1+\alpha)c\right)}{\alpha  \kappa_1 \left(\kappa_1^2-c\right)}\\
&=\frac{\left(\kappa_1^2+c\right) \left(2 \kappa_1^4-(\alpha +7) \kappa_1^2 c-(2 \alpha +1) c^2\right)}{\alpha  \kappa_1 \left(\kappa_1^2-c\right)}
\end{split}
\end{equation*}
  Equating the two previous  expressions of $e_{22}(\kappa_1)$, we get 
  $$ (2 \alpha +5) \kappa_1^4+2 (\alpha +3) \kappa_1^2 c+(2 \alpha +1) c^2  =0.$$
   This is a polynomial on $\kappa_1$, whose coefficients are constant and not all zero.  This implies that $\kappa_1$ is constant, obtaining a contradiction.
 \item Case $e_1(\kappa_1)\not=0$.
 We substitute \eqref{d7} in the expression of $K$ in \eqref{bb2}, obtaining
\begin{equation}\label{pe1}
P_1e_1(\kappa_1)^2+Q_1e_2(\kappa_1)^2+R_1=0,
\end{equation}
where
\begin{equation*}
\begin{split}
P_1&=\frac{\alpha  \kappa_1^2+(\alpha +4) c}{(\alpha +1) \kappa_1^2+c},\\
Q_1&= \frac{(\alpha +4) \kappa_1^2+\alpha  c}{\kappa_1^2+(1+\alpha)c},\\
R_1&= \kappa_1^4+\frac{\left(\kappa_1^2+c\right)^2}{\alpha }+c^2.
\end{split}
\end{equation*}

We differentiate \eqref{pe1} with respect to $e_1$ obtaining
$$e_1(P_1)e_1(\kappa_1)^2+e_1(Q_1)e_2(\kappa_1)^2+e_1(R_1)+2P_1e_1(\kappa_1)e_{11}(\kappa_1)+2Q_1 e_2(\kappa_1)e_{12}(\kappa_1)=0.$$
Simplifying by $e_1(\kappa_1)$ because $e_1(\kappa_1)\not=0$, we obtain 
\begin{equation}\label{pe2}
P_2e_1(\kappa_1)^2+Q_2e_2(\kappa_1)^2+R_2=0,
\end{equation}
 where 
\begin{equation*}
\begin{split}
P_2&= \frac{2 (\alpha +2) \kappa_1 \left(\alpha  \kappa_1^4+(2-3 \alpha ) k^2 c-2 (\alpha +5) c^2\right)}{\left(\kappa_1^2-c\right) \left((\alpha +1) \kappa_1^2+c\right)^2},\\
Q_2&= \frac{2 (\alpha +2) \kappa_1 \left(2 (\alpha +1) \kappa_1^6+\left(\alpha ^2-3 \alpha -8\right) \kappa_1^4 c-\left(3 \alpha ^2+12 \alpha +10\right) \kappa_1^2 c^2-\alpha  (2 \alpha +3) c^3\right)}{\left(\kappa_1^2-c\right) \left(\kappa_1^2+(\alpha+1)c\right)^2 \left((\alpha +1) \kappa_1^2+c\right)} ,\\
R_2&=\frac{2 (\alpha +2) \kappa_1^5-8 (\alpha +1) \kappa_1^3 c-2 (\alpha +6) c\kappa_1^2}{\alpha  \left(\kappa_1^2-c\right)} .
\end{split}
\end{equation*}
Before the next computations, we explain how we will arrive to a contradiction.  From Eqs. \eqref{pe1} and \eqref{pe2} we get   the values of $e_1(\kappa_1)^2$ and $e_2(\kappa_1)^2$. Here we   need to discuss if the determinant of the coefficients of $e_1(\kappa_1)^2$ and $e_2(\kappa_1)^2$, namely, $P_1Q_2-P_2Q_1$, is $0$ or not. If $P_1Q_2-P_2Q_1=0$, then a suitable combination of \eqref{pe1} and \eqref{pe2} gives a linear combination of $R_1$ and $R_2$ which will be a polynomial on $\kappa_1$ with constant coefficients. This proves that $\kappa_1$ is constant, which it is a contradiction. If  $P_1Q_2-P_2Q_1\not= 0$, and with the values of $e_1(\kappa_1)^2$ and $e_2(\kappa_1)^2$, we differentiate $e_1(\kappa_1)^2$ with respect to $e_1$. This gives an identity involving $e_{11}(\kappa_1)$, $e_{12}(\kappa_1)$, $e_1(\kappa_1)^2$ and $e_2(\kappa_1)^2$. Substituting all these functions by their  expressions, we finally obtain a polynomial on $\kappa_1$ with constant coefficients, obtaining the desired   contradiction. 

 We have $P_1Q_2-P_2Q_1=(\alpha^2 -4)   \kappa_1 \left(\kappa_1^2-c\right)^3$. Notice that $\kappa_1^2-c\not=0$.
\begin{enumerate}
\item Case $\alpha^2-4=0$.   If $\alpha=-2$, the linear combinations $P_2\eqref{pe1}-P_1 \eqref{pe2}=0$ gives $8c\kappa_1=0$, which it is not possible. If $\alpha=2$, the same  combination of both equations yields $8 c\kappa_1 \left(-5 \kappa_1^4+6 \kappa_1^2 c+15 c^2\right)=0$, which implies that $\kappa_1$ is constant, obtaining a contradiction. 
  
 \item Case $\alpha^2-4\not=0$.  Solving \eqref{pe1} and \eqref{pe2}, we obtain 
\begin{equation}\label{es2}
\begin{split}
e_1(\kappa_1)^2&:=Z_1=-\frac{M_1\left((\alpha +1) \kappa_1^2+c\right)}{\alpha ^2 \left(\alpha ^2-4\right) \left(\kappa_1^2-c\right)^3},\\
e_2(\kappa_1)^2&:=Z_2=\frac{M_2 c \left(\kappa_1^2+(1+\alpha)c\right)^2}{\alpha ^2 \left(\alpha ^2-4\right) \left(\kappa_1^2-c\right)^3},
\end{split}
\end{equation}
where 
\begin{equation*}
\begin{split}
M_1&=    (\alpha +1) (\alpha^2 -4) \kappa_1^8-(\alpha  (\alpha  (4 \alpha +11)+16)+12) \kappa_1^6 c\\
&+(\alpha  (\alpha  (7 \alpha +13)+4)-12) \kappa_1^4 c^2+(\alpha  (\alpha  (4 \alpha  (\alpha +5)+39)+16)-4) \kappa_1^2 c^3\\
&+2 \alpha ^2 (\alpha +1) (\alpha +3) c^4,\\
M_2&=   -(\alpha  (5 \alpha +8)+4) \kappa_1^4+2 (\alpha  (\alpha  (2 \alpha +5)-4)-4) \kappa_1^2 c+(\alpha  (\alpha  (2 \alpha +15)+24)-4) c^2,
\end{split}
\end{equation*}
Differentiating $e_1(\kappa_1)^2$ with respect to $e_1$, we obtain 
$$2 e_1(\kappa_1)e_{11}(\kappa_1)=e_1(\kappa_1)\frac{d}{d\kappa_1}\left(Z_1\right).$$
  Simplifying by $e_1(\kappa_1)$ and using \eqref{d7} and \eqref{es2}, we obtain 
$$ \kappa_1^2 \left(\kappa_1^2+(1+\alpha)c\right) \left((\alpha +4) \kappa_1^2+\alpha  c\right)=0.$$
This implies that $\kappa_1$ is constant, which it is a contradiction. This completes the proof of Thm. \ref{t1}.
\end{enumerate}
 \end{enumerate}

\section{Proof of Theorem \ref{t22}}\label{s5}

It remains to prove the converse. Let $\Sigma$ be an $\alpha$-singular minimal surface with a constant principal curvature.    Without loss of generality, assume that  the principal curvature $\kappa_2$ is constant. Let $\kappa_2=c$. If $c=0$, then $K=0$ and Thm. \ref{t1} concludes that the surface is a plane or a cylindrical  singular minimal  surface. This proves Thm. \ref{t22} in this situation. 
 
 Suppose now $c\not=0$. If the principal curvature   $\kappa_1$ is constant, then the surface is isoparametric, hence it is a plane or a sphere. This proves Thm. \ref{t22} in this situation.
 
 Suppose that $\kappa_1$ is not constant and we will arrive to a contradiction. We have $K=c\kappa_1$ and $H=\kappa_1+c$. Identities \eqref{b1} gives $\omega(e_2)=0$. The expression of $K$ in \eqref{bb} is
 \begin{equation}\label{f1}
 c\kappa_1=\frac{e_{22}(\kappa_1)}{\kappa_1-c}-2\frac{e_2(\kappa_1)^2}{(\kappa_1-c)^2}. 
 \end{equation}
The last two equations of \eqref{d3} are
 \begin{equation*}
 \begin{split}
 0&=\frac{e_1(\kappa_1)\langle\Phi,\vec{a}\rangle+(\kappa_1+c)\gamma}{\alpha}+\gamma \kappa_1,\\
 0&=\frac{e_2(\kappa_1)\langle\Phi,\vec{a}+(\kappa_1+c)\mu}{\alpha}+\mu c.
 \end{split}
 \end{equation*}
 We obtain the expressions the functions $\gamma$ and $\mu$, namely,  
 $$\gamma=-\frac{e_1(\kappa_1)\langle\Phi,\vec{a}\rangle}{c+(1+\alpha)\kappa_1},\quad \mu=-\frac{e_2(\kappa_1)\langle\Phi,\vec{a}\rangle}{\kappa_1+(1+\alpha)c}.$$
  Differentiating $\mu$ with respect to $e_2$, we have
\begin{equation*}
e_2(\mu)=-\frac{\langle\Phi,\vec{a}\rangle}{\kappa_1+(1+\alpha)c}e_{22}(\kappa_1)+  \frac{2\langle\Phi,\vec{a}\rangle}{((1+\alpha)c+\kappa_1)^2}e_2(\kappa_1)^2.
\end{equation*}
On the other hand, the fourth equation of \eqref{d3} gives
$$e_2(\mu)=c\langle\Phi,\vec{a}\rangle \frac{c+\kappa_1}{\alpha},$$
because $\omega(e_2)=0$. From the last two expressions of $e_2(\mu)$ we obtain 
\begin{equation}\label{E1}
e_{22}(\kappa_1)=((\alpha+1)c+\kappa_1) \left(\frac{2 e_2(\kappa_1)^2}{((\alpha+1)c+\kappa_1)^2}-\frac{c (c+\kappa_1)}{\alpha }\right).
\end{equation}
 Inserting in \eqref{f1}, we have 
 $$\frac{(c-\kappa_1) \left(\alpha  \left(c^2+\kappa_1^2\right)+(c+\kappa_1)^2\right)}{\alpha }-\frac{2 (\alpha +2) e_2(\kappa_1)^2}{(\alpha+1)c+\kappa_1}=0.$$
 If $\alpha=-2$, then this equation is $(c+\kappa_1)^2-2 \left(c^2+\kappa_1^2\right)=0$, which implies that $\kappa_1$ is constant. If $\alpha\not=-2$, then 
\begin{equation}\label{E2}
e_2(\kappa_1)^2=\frac{(c-\kappa_1) ((\alpha+1)c+\kappa_1) \left((\alpha +1) c^2+2 c \kappa_1+(\alpha +1)\kappa_1^2\right)}{2 \alpha  (\alpha +2)}.
\end{equation}
If $e_2(\kappa_1)=0$ identically, then \eqref{E2} implies $(\alpha+1)c+\kappa_1=0$ or $(\alpha +1) c^2+2 c \kappa_1+(\alpha +1)\kappa_1^2=0$, hence $\kappa_1$ is constant. In both cases, we deduce that $\kappa_1$ is a constant, which it is a contradiction. Thus $e_2(\kappa_1)\not=0$.  Differentiate  with respect to $e_2$ in \eqref{E2}, and   simplify later  by $e_2(\kappa_1)$, we have 
 $$2 e_{22}(\kappa_1)= \frac{\left(-\alpha ^2+\alpha +2\right) c^3+2 (\alpha -1) \alpha  c^2 \kappa_1-3 \left(\alpha ^2+\alpha +2\right) c \kappa_1^2-4 (\alpha +1) \kappa_1^3}{2 \alpha  (\alpha +2)}.$$
Using the value of $e_{22}(\kappa_1)$ and $e_2(\kappa_1)^2$ in \eqref{E1} and \eqref{E2}, respectively, we have 
 $$-(\alpha -1) \kappa_1^2+2 (\alpha +1) c \kappa_1+(\alpha +1) c^2=0.$$
 This proves that $\kappa_1$ is constant, which it is a contradiction.
\section{Proof of Theorem \ref{t3}}\label{s6}

Since planes parallel to $\vec{a}$ and spheres centered at $\Pi$ are singular minimal surfaces with constant mean curvature, it remain to prove the converse. 

  Let $\Sigma$ be an $\alpha$-singular minimal surface with  constant mean curvature $H$.     Instead to use the previous computations of Sect. \ref{s2}, we do a direct proof. If a principal curvature is constant, then the other one is also constant because $H$ is constant. In such a case, the surface is isoparametric, hence   $\Sigma$ is a plane or a sphere. This proves the theorem in this situation. 
  
  From now, we suppose that both principal curvatures are not constant.   Let  $\Omega\subset\Sigma$ be an open set of $\Sigma$ formed by points $p$ where $ N(p)\not=\pm \vec{a}$. This open exists because $\Sigma$ is not a plane. We also assume that $\Omega$ is not formed by umbilic points because $\Sigma$ is not a plane neither a sphere. 
    The singular minimal surface equation \eqref{eq1} writes as 
$$H\langle\Phi(p),\vec{a}\rangle-\alpha\langle N(p),\vec{a}\rangle=0,\quad p\in \Omega.$$
Let $\{e_1,e_2\}$ be an orthonormal frame in $\Omega$ formed by principal directions, with $Se_i=\kappa_i e_i$. Differentiating the above identity  with respect to $e_1$ and $e_2$, we obtain 
$$(H+\alpha\kappa_1)\langle e_1,\vec{a}\rangle=0$$
$$(H+\alpha\kappa_2)\langle e_2,\vec{a}\rangle=0.$$
Since $N\not=\pm\vec{a}$ in $\Omega$, then $\langle e_1,\vec{a}\rangle\not=0$ or $\langle e_2,\vec{a}\rangle\not=0$ at some point   $p_0\in \Omega$. Without loss of generality, we suppose that $\langle e_1,\vec{a}\rangle\not=0$ at $p_0$. Then in an open set $\widetilde{\Omega}\subset\Omega$ around $p_0$ we have   $H+\alpha\kappa_1=0$. This proves that $\kappa_1$  is constant, which it is a contradiction.

\subsection*{Acknowledgment}
The author  has been partially supported by MINECO/MICINN/FEDER grant no. PID2023-150727NB-I00,  and by the ``Mar\'{\i}a de Maeztu'' Excellence Unit IMAG, reference CEX2020-001105- M, funded by MCINN/AEI/10.13039/ 501100011033/ CEX2020-001105-M.

\end{document}